\documentclass[11pt]{article}

\usepackage[a4paper,margin=3cm]{geometry}
\usepackage{amsmath,amssymb,amsthm,mathtools}
\usepackage{bm}
\usepackage{booktabs}
\usepackage{enumitem}
\usepackage{graphicx}
\usepackage{microtype}
\usepackage[numbers,sort&compress]{natbib}
\usepackage{tabularx}
\usepackage{xcolor}
\usepackage[hidelinks]{hyperref}

\title{Averaged Extensions of Golomb's Triangular Recursion:\\
Critical Invariance and Supercritical Constraints}
\author{
Marco Mantovanelli\\
Independent Researcher\\[1mm]
\texttt{marco@mantovanelli.de}\\[1mm]
\href{https://orcid.org/0009-0002-0631-293X}{ORCID: 0009-0002-0631-293X}
}
\date{}

\hypersetup{
  pdftitle={Averaged Extensions of Golomb's Triangular Recursion: Critical Invariance and Supercritical Constraints},
  pdfauthor={Marco Mantovanelli},
  pdfsubject={Averaged meta-Fibonacci recursions},
  pdfkeywords={meta-Fibonacci recursion, Golomb triangular recursion, nested recursion, slow growth, power mean}
}

\newtheorem{theorem}{Theorem}[section]
\newtheorem{lemma}[theorem]{Lemma}
\newtheorem{proposition}[theorem]{Proposition}
\newtheorem{corollary}[theorem]{Corollary}
\theoremstyle{definition}
\newtheorem{definition}[theorem]{Definition}
\newtheorem{conjecture}[theorem]{Conjecture}
\theoremstyle{remark}
\newtheorem{remark}[theorem]{Remark}

\newcommand{\floor}[1]{\left\lfloor #1\right\rfloor}
\newcommand{\R}{\mathbb{R}}

\begin{document}

\maketitle

\begin{abstract}
For an integer $m\ge 1$ and a parameter $\alpha>0$, consider the nested
recursion
\[
Q_{\alpha,m}(n)
=
1+\floor{
\frac{\alpha}{m}\sum_{j=1}^{m}
Q_{\alpha,m}\bigl(n-Q_{\alpha,m}(n-j)\bigr)
},
\qquad n>m,
\]
with $Q_{\alpha,m}(1)=\cdots=Q_{\alpha,m}(m)=1$.
For $m=1$ and $\alpha=1$, this is Golomb's non-homogeneous triangular
recursion.  We prove that its canonical triangular solution is preserved,
up to an initial index shift, by every finite arithmetic averaging length.
More generally, the same exact solution is generated by any aggregator that
obeys a simple local floor-lock condition.  This includes all power means of
finite order, including the harmonic and geometric cases, as well as the
minimum and positively weighted quasi-arithmetic means.  The maximum, which
lies outside this class for $m\ge2$, has a different but explicit block law.
Consequently, every value $k\ge2$ occurs
exactly $k$ times in the floor-admissible class, and
\[
Q_{1,m}(n)
=
\floor{\frac{1+\sqrt{1+8(n-m)}}{2}}
\sim \sqrt{2n}
\qquad(n>m).
\]
For $0<\alpha<1$, the sequence is identically one.  Near criticality, for
$\alpha=1+\delta$ with $0<\delta<(2m-1)^{-1}$, we determine the exact first
departure time, which is of order $\delta^{-2}$.  We also prove a finite-step
breakdown criterion for large $\alpha$ and a conditional slope theorem:
any globally defined solution with a limiting density in $(0,1)$ must have
slope $1-\alpha^{-1}$.  Exact-arithmetic computations support, but do not
prove, a supercritical linear-growth regime.
\end{abstract}



\section{Introduction}

Meta-Fibonacci recursions are integer recursions whose arguments depend on
earlier values of the sequence itself.  Hofstadter's $Q$-sequence,
\[
Q(n)=Q(n-Q(n-1))+Q(n-Q(n-2)),
\]
is the best-known example \cite{hofstadter}.  Its apparently irregular
behavior contrasts with several structured nested recursions studied in the
literature, including slow and monotone families
\cite{tanny,higham-tanny,balamohan-kuznetsov-tanny,callaghan-chew-tanny},
tree-solvable and quasi-polynomial families
\cite{isgur-kuznetsov-tanny,celaya-ruskey,fox,fox-symbolic}, and diluted
variants \cite{deane-gentile,deane-gentile-subsets}.

The starting point of the present work is Golomb's non-homogeneous triangular
recursion
\begin{equation}\label{eq:golomb-triangular}
g(n)=1+g(n-g(n-1)),
\qquad g(1)=1.
\end{equation}
Its canonical slow solution is
\[
1,\;2,2,\;3,3,3,\;4,4,4,4,\ldots,
\]
so the frequency of the value $k$ is exactly $k$.  This solution and its
closed form are classical in the theory of the Golomb recursion
\cite{barbeau-chew-tanny,isgur-kuznetsov-tanny,sunohara-tanny}.
It should not be confused with Golomb's self-describing sequence
\cite{vardi,petermann}, in which the frequency of $k$ is the value of the
sequence at $k$; recent finite-memory variants of that different object are studied in
\cite{cloitre-almost-golomb,cloitre-beatty}.

We study the two-parameter averaged extension
\begin{equation}\label{eq:main-recursion-intro}
Q_{\alpha,m}(n)
=
1+\floor{
\frac{\alpha}{m}\sum_{j=1}^{m}
Q_{\alpha,m}\bigl(n-Q_{\alpha,m}(n-j)\bigr)
},
\end{equation}
initialized by $m$ ones.  At $m=1$ and $\alpha=1$,
equation~\eqref{eq:main-recursion-intro} reduces exactly to
\eqref{eq:golomb-triangular}.  The triangular block itself is therefore tied
to a known Golomb recursion.  What is new is its exact robustness under
arbitrarily long finite averaging and general floor-admissible aggregation,
together with a separate exact law for the maximum when $m\ge2$ and an exact
first supercritical departure time.

This paper makes four principal contributions.
\begin{enumerate}[label=\textup{(\roman*)},leftmargin=2.2em]
\item At $\alpha=1$, arithmetic averaging of every finite length preserves
the triangular orbit exactly, up to the prescribed initial shift.
\item The same conclusion holds for every floor-admissible aggregator,
including all power means of finite order and the minimum.
\item For $m\ge2$, the maximum aggregator lies outside the floor-admissible
class but has a different, completely explicit transient and block law; for
$m=1$, it is already covered by the invariance theorem.
\item For $\alpha=1+\delta$ with $\delta>0$ small, the exact first departure
from the critical orbit occurs at a computable index of order
$(2\delta^2)^{-1}$.
\end{enumerate}

In addition, the subcritical solution is proved to be identically one,
large parameters are shown to terminate in finite time, and any limiting
supercritical density in $(0,1)$ is forced to equal $1-\alpha^{-1}$.
Reproducible exact-arithmetic computations complement these results.

Only the subcritical and critical regimes, the finite shadowing theorem, and
the conditional slope constraint are proved here.  Global existence and
convergence in a supercritical interval remain open.

\begin{table}[htbp]
\centering
\small
\caption{Logical status of the principal statements.}
\label{tab:status}
\begin{tabularx}{\textwidth}{@{}
>{\raggedright\arraybackslash}p{0.35\textwidth}
>{\raggedright\arraybackslash}p{0.17\textwidth}
>{\raggedright\arraybackslash}X@{}}
\toprule
Statement & Status & Qualification \\
\midrule
Frozen solution for $0<\alpha<1$ & proved & Global and unconditional. \\
Triangular blocks at $\alpha=1$ & proved & Global; valid for every
floor-admissible aggregator. \\
Closed form, shift identity, and maximum-aggregator law & proved & Exact
formulas, including the transient. \\
Near-critical shadowing & proved & Exact first departure for
$0<\alpha-1<(2m-1)^{-1}$. \\
Supercritical slope constraint & conditionally proved & Assumes global
existence and an interior positive liminf and limsup. \\
Supercritical existence and convergence & open & Finite exact-arithmetic
computations provide evidence only. \\
\bottomrule
\end{tabularx}
\end{table}

The paper is organized as follows.  Section~\ref{sec:definition} introduces
the recursion and elementary phase bounds.  Section~\ref{sec:critical}
proves critical invariance for general aggregators.  Section~\ref{sec:super}
contains the shadowing and slope theorems.  Section~\ref{sec:numerics}
reports certified finite computations, and Section~\ref{sec:conclusion}
summarizes the open problems.

\section{Definition and elementary phase bounds}\label{sec:definition}

Fix an integer $m\ge1$ and a real parameter $\alpha>0$.  Set
\begin{equation}\label{eq:initial-data}
Q_{\alpha,m}(1)=\cdots=Q_{\alpha,m}(m)=1,
\end{equation}
and, for $n>m$, define
\begin{equation}\label{eq:main-recursion}
Q_{\alpha,m}(n)
=
1+\floor{
\frac{\alpha}{m}\sum_{j=1}^{m}
Q_{\alpha,m}\bigl(n-Q_{\alpha,m}(n-j)\bigr)
}.
\end{equation}
The recursion is well-defined at $n$ if
\begin{equation}\label{eq:admissible-indices}
1\le n-Q_{\alpha,m}(n-j)<n
\qquad (1\le j\le m),
\end{equation}
so every recursive argument refers to an already constructed term.

For $m=1$, equation~\eqref{eq:main-recursion} becomes
\[
Q_{\alpha,1}(n)
=1+\floor{\alpha
Q_{\alpha,1}\bigl(n-Q_{\alpha,1}(n-1)\bigr)}.
\]
At $\alpha=1$, this is equation~\eqref{eq:golomb-triangular}; the floor is
then redundant because its argument is an integer.

\subsection{The frozen phase}

\begin{proposition}\label{prop:subcritical}
If $0<\alpha<1$, then equation~\eqref{eq:main-recursion} is globally
well-defined and
\[
Q_{\alpha,m}(n)=1
\qquad (n\ge1).
\]
\end{proposition}

\begin{proof}
The initial values are one.  Suppose that $Q_{\alpha,m}(r)=1$ for every
$r<n$.  Then each outer value $Q_{\alpha,m}(n-j)$ equals one, every recursive
argument is $n-1$, and every term in the mean equals one.  Hence
\[
Q_{\alpha,m}(n)=1+\floor{\alpha}=1.
\]
The argument $n-1$ lies in $\{1,\ldots,n-1\}$, so the induction also proves
global well-definedness.
\end{proof}

\subsection{An elementary breakdown bound}

For large parameters, failure occurs immediately.

\begin{proposition}\label{prop:large-alpha-breakdown}
If $\alpha\ge m+1$, then the recursion fails to be well-defined at
$n=m+2$.
\end{proposition}

\begin{proof}
At the first recursive step all inner values equal one, so
\[
Q_{\alpha,m}(m+1)=1+\floor{\alpha}.
\]
At $n=m+2$, the branch $j=1$ requests the index
\[
m+2-Q_{\alpha,m}(m+1)
=m+1-\floor{\alpha}\le0.
\]
Thus condition~\eqref{eq:admissible-indices} fails.
\end{proof}

Proposition~\ref{prop:large-alpha-breakdown} gives only a coarse upper
bound.  The geometry of the full set of globally admissible parameters is
unknown and is not assumed here to be an interval.

\section{Critical invariance under general means}\label{sec:critical}

The critical proof uses only a local property of the arithmetic mean.  It is
therefore natural to isolate that property.

\begin{definition}\label{def:floor-admissible}
Let $m\ge1$.  A map
\[
\mathcal{M}:(0,\infty)^m\longrightarrow(0,\infty)
\]
is \emph{floor-admissible} if, for every integer $r\ge1$ and every
$\bm{x}\in\{r,r+1\}^m$ having at least one component equal to $r$,
\begin{equation}\label{eq:floor-lock}
r\le \mathcal{M}(\bm{x})<r+1.
\end{equation}
\end{definition}

Given a floor-admissible $\mathcal{M}$, define $Q_{\mathcal{M},m}$ by $m$
initial ones and
\begin{equation}\label{eq:aggregator-recursion}
Q_{\mathcal{M},m}(n)
=
1+\floor{
\mathcal{M}\bigl(Y_1(n),\ldots,Y_m(n)\bigr)
},
\end{equation}
where
\begin{equation}\label{eq:branches}
Y_j(n)
=
Q_{\mathcal{M},m}\bigl(n-Q_{\mathcal{M},m}(n-j)\bigr).
\end{equation}

The arithmetic mean is floor-admissible.  More generally, so is every finite
power mean
\[
M_p(x_1,\ldots,x_m)
=
\left(\frac1m\sum_{j=1}^{m}x_j^p\right)^{1/p},
\qquad p\in\R\setminus\{0\},
\]
together with the geometric mean $M_0=(\prod_jx_j)^{1/m}$ and the minimum
$M_{-\infty}=\min_jx_j$.  Indeed, each of these means lies between the
minimum and maximum of its inputs and is strictly below the maximum when the
inputs are not all maximal.  More explicitly, if
$\varphi:(0,\infty)\to\mathbb{R}$ is continuous and strictly monotone and
$w_j>0$ with $\sum_{j=1}^m w_j=1$, then the weighted quasi-arithmetic mean
\[
M_{\varphi,\bm w}(x_1,\ldots,x_m)
=
\varphi^{-1}\!\left(\sum_{j=1}^m w_j\varphi(x_j)\right)
\]
is strictly internal and hence floor-admissible.  For $m\ge2$, the maximum
$M_{+\infty}$ is not floor-admissible.  For $m=1$, it coincides with the
identity and is already covered by the theorem below.

For every $k\ge1$, put
\begin{equation}\label{eq:block-endpoints}
s_k=m+\frac{k(k-1)}2,
\qquad
e_k=m+\frac{k(k+1)}2-1,
\qquad
I_k=[s_k,e_k]\cap\mathbb{Z}.
\end{equation}
Thus $I_1=\{m\}$ is an auxiliary interval containing the final initial one;
it is not meant to replace the full initial block $\{1,\ldots,m\}$.
For $k\ge2$, the interval $I_k$ has length $k$.

\begin{theorem}[Critical aggregator invariance]\label{thm:aggregator}
Let $m\ge1$ and let $\mathcal{M}$ be floor-admissible.  Then
equation~\eqref{eq:aggregator-recursion} is globally well-defined and
\begin{equation}\label{eq:exact-blocks}
Q_{\mathcal{M},m}(n)=k
\quad\Longleftrightarrow\quad
n\in I_k
\qquad (k\ge2).
\end{equation}
In particular, each value $k\ge2$ occurs exactly $k$ consecutive times.
\end{theorem}

\begin{proof}
Write $Q=Q_{\mathcal{M},m}$.  We use induction over the blocks $I_k$ and,
inside each block, induction from left to right.  Suppose that all values
before
\[
n=s_k+a,
\qquad 0\le a\le k-1,
\]
have been constructed and agree with the asserted blocks.  This includes
the initial step $k=2,a=0$.

For $1\le j\le m$, set
\[
r_j=Q(n-j),
\qquad
x_j=n-r_j.
\]
Since $n>m$, we have $n-j\ge1$, and since $n-j<n$, the outer value $r_j$
is already known.  Moreover, $1\le r_j\le k$.

We first locate $x_j$.  If $a=0$, then $n-j<s_k$, so $r_j\le k-1$ and
\[
x_j\ge s_k-(k-1)=s_{k-1}.
\]
If $a\ge1$, then $r_j\le k$ and
\[
x_j\ge s_k+a-k=s_{k-1}+a-1\ge s_{k-1}.
\]
In both cases,
\[
x_j\le n-1\le e_k.
\]
Since $e_{k-1}=s_k-1$, the intervals $I_{k-1}$ and $I_k$ are adjacent.
Therefore
\[
x_j\in I_{k-1}\cup I_k.
\]
Also $x_j<n$, so its value is already known, and hence
\begin{equation}\label{eq:two-level-branches}
Q(x_j)\in\{k-1,k\}.
\end{equation}

At least one branch has the lower value.  Take $j=1$.  If $a=0$, then
$n-1=e_{k-1}$ and
\[
n-Q(n-1)=s_k-(k-1)=s_{k-1}.
\]
If $a\ge1$, the inner induction gives $Q(n-1)=k$, so
\[
n-Q(n-1)=s_k+a-k=s_{k-1}+a-1\in I_{k-1}.
\]
Thus in both cases
\begin{equation}\label{eq:lower-branch}
Q\bigl(n-Q(n-1)\bigr)=k-1.
\end{equation}

The vector of branch values therefore belongs to $\{k-1,k\}^m$ and has at
least one component $k-1$.  Floor-admissibility gives
\[
k-1
\le
\mathcal{M}\bigl(Y_1(n),\ldots,Y_m(n)\bigr)
<k,
\]
and equation~\eqref{eq:aggregator-recursion} yields $Q(n)=k$.  This closes
both inductions.  The bounds $1\le x_j<n$ established at every step prove
global well-definedness at the same time.
\end{proof}

\begin{corollary}[Arithmetic critical recursion]\label{cor:arithmetic}
For every $m\ge1$, the recursion $Q_{1,m}$ is globally well-defined and the
value $k\ge2$ occurs exactly on $I_k$.
\end{corollary}

\begin{corollary}[Closed form and shift identity]\label{cor:closed-form}
For $n>m$,
\begin{equation}\label{eq:closed-form}
Q_{1,m}(n)
=
\floor{\frac{1+\sqrt{1+8(n-m)}}2}.
\end{equation}
If $g$ denotes the canonical solution of
equation~\eqref{eq:golomb-triangular}, then
\begin{equation}\label{eq:shift-identity}
Q_{1,m}(m+r)=g(r+1)
\qquad (r\ge0).
\end{equation}
Consequently,
\begin{equation}\label{eq:critical-asymptotic}
Q_{1,m}(n)=\sqrt{2(n-m)}+O(1)\sim\sqrt{2n}.
\end{equation}
\end{corollary}

\begin{proof}
Let $t=n-m\ge1$.  By Theorem~\ref{thm:aggregator}, $Q_{1,m}(n)=k$ if and
only if
\[
\frac{k(k-1)}2\le t<\frac{k(k+1)}2.
\]
Inverting these triangular inequalities gives
equation~\eqref{eq:closed-form}.  Equation~\eqref{eq:shift-identity} follows
because the same inequalities describe the canonical Golomb triangular
solution.  The asymptotic estimate is immediate.
\end{proof}

The independence of $m$ in equation~\eqref{eq:critical-asymptotic} is thus
not merely asymptotic.  The entire nontrivial part of the critical sequence
is an exact horizontal translate of the $m=1$ solution.

We will need a more precise form of the branch rigidity.  Put
\begin{equation}\label{eq:critical-branch-average}
B_m(n)
=
\frac1m\sum_{j=1}^{m}
Q_{1,m}\bigl(n-Q_{1,m}(n-j)\bigr).
\end{equation}

\begin{lemma}[Exact critical branch average]\label{lem:branch-average}
If $n=s_k+a$, where $k\ge2$ and $0\le a\le k-1$, then
\begin{equation}\label{eq:branch-count}
B_m(n)=k-1+\frac{c_{m,k,a}}m,
\end{equation}
where $c_{m,k,0}=0$ and, for $1\le a\le k-1$,
\begin{equation}\label{eq:c-count}
c_{m,k,a}
=
\max\left\{0,\,m+\frac{a(a-1)}2-\frac{k(k-1)}2\right\}.
\end{equation}
In particular,
\begin{equation}\label{eq:branch-collapse}
B_m(n)=k-1
\qquad (k\ge m+1).
\end{equation}
\end{lemma}

\begin{proof}
The proof of Theorem~\ref{thm:aggregator} shows that every summand in
$B_m(n)$ is either $k-1$ or $k$.  For $a\ge1$, a summand equals $k$
precisely when
\[
Q_{1,m}(n-j)\le a.
\]
By the critical block law, $Q_{1,m}(t)\le a$ precisely when
$t\le e_a=m+a(a+1)/2-1$.  Thus the number of upper summands is the number of
integers in
\[
[n-m,n-1]\cap(-\infty,e_a],
\]
namely
\[
\max\left\{0,m+\frac{a(a+1)}2-\frac{k(k-1)}2-a\right\},
\]
which is equation~\eqref{eq:c-count}.  When $a=0$, no upper summand is
possible.  Finally, if $k\ge m+1$, then
\[
\frac{k(k-1)}2-\frac{a(a-1)}2\ge k-1\ge m,
\]
so $c_{m,k,a}=0$.
\end{proof}

\subsection{The boundary case of the maximum}

For $m\ge2$, the strict upper inequality in
Definition~\ref{def:floor-admissible} excludes the maximum.  The trivial case
$m=1$ is included below for completeness.  For $m\ge2$, the maximum does not
preserve the initial shift, but it still admits a complete description.  Let
$T_r=r(r+1)/2$.

\begin{theorem}[Maximum aggregator]\label{thm:maximum}
Let $H_m(1)=\cdots=H_m(m)=1$ and, for $n>m$, define
\begin{equation}\label{eq:max-recursion}
H_m(n)
=
1+\max_{1\le j\le m}H_m\bigl(n-H_m(n-j)\bigr).
\end{equation}
For $k\ge m+1$, put
\begin{equation}\label{eq:max-start}
u_k=2m+T_{k-1}-T_m.
\end{equation}
Then equation~\eqref{eq:max-recursion} is globally well-defined and
\begin{equation}\label{eq:max-explicit}
H_m(n)=
\begin{cases}
1,&1\le n\le m,\\
n-m+1,&m+1\le n\le2m-1,\\
k,&u_k\le n\le u_k+k-1,\quad k\ge m+1.
\end{cases}
\end{equation}
The middle range is empty for $m=1$.  Equivalently, for $n\ge2m$,
\begin{equation}\label{eq:max-closed-form}
H_m(n)
=
\floor{\frac{1+\sqrt{1+8(n-2m+T_m)}}2},
\end{equation}
and hence $H_m(n)\sim\sqrt{2n}$.
\end{theorem}

\begin{proof}
Let $h$ be the right-hand side of equation~\eqref{eq:max-explicit}.  It is
nondecreasing and satisfies $1\le h(t)\le t$.  Hence all recursive arguments
are admissible.  Since $h(n-j)\ge h(n-m)$ for $1\le j\le m$, monotonicity
also gives
\[
\max_{1\le j\le m}h\bigl(n-h(n-j)\bigr)
=h\bigl(n-h(n-m)\bigr).
\]
It remains to verify
\begin{equation}\label{eq:max-reduced}
h(n)=1+h\bigl(n-h(n-m)\bigr).
\end{equation}

For $n=m+r$, $1\le r\le m-1$, one has $h(n-m)=1$ and
$h(n-1)=r$, proving equation~\eqref{eq:max-reduced}.  Next write
$n=2m+a$, $0\le a\le m$.  Then $h(n-m)=a+1$, so the inner argument is
$2m-1$, where $h(2m-1)=m$.  This proves the first full block, of value
$m+1$.

Now take $k\ge m+2$ and write $n=u_k+a$, $0\le a\le k-1$.  Set
$x=n-h(n-m)$.  If $a=0$, then $h(n-m)\le k-1$ and
\[
x\ge u_k-(k-1)=u_{k-1}.
\]
If $a\ge1$, then $n-m<u_{k+1}$, so $h(n-m)\le k$ and again
$x\ge u_k+a-k\ge u_{k-1}$.

The first occurrence of the value $a+1$ is no later than $n-m$.  This is
immediate for $a<m$.  For $a\ge m$, it follows from
\[
u_k+a-m-u_{a+1}
=T_{k-1}-T_a+a-m\ge0;
\]
the right-hand side is minimized at $a=k-1$, where it equals $k-1-m$.
Consequently $h(n-m)\ge a+1$ and $x\le u_k-1$.  Thus
$x\in[u_{k-1},u_k-1]$, the block of value $k-1$, which proves
equation~\eqref{eq:max-reduced}.  Formula~\eqref{eq:max-closed-form} and the
asymptotic follow from $u_{k+1}=u_k+k$.
\end{proof}

\section{Near-critical shadowing and supercritical constraints}\label{sec:super}

We now return to the arithmetic recursion~\eqref{eq:main-recursion} with
$\alpha>1$.  The following theorem gives a rigorous bridge from the critical
orbit to supercritical dynamics.

\subsection{Exact first departure from criticality}

\begin{theorem}[Exact near-critical shadowing]\label{thm:shadowing}
Let $m\ge1$, write $\alpha=1+\delta$, and assume
\begin{equation}\label{eq:delta-small}
0<\delta<\frac{1}{2m-1}.
\end{equation}
Define
\begin{equation}\label{eq:departure-block}
h_\delta=\left\lceil\frac1\delta\right\rceil,
\qquad
K_\delta=h_\delta+1,
\qquad
n_\delta=s_{K_\delta}
=m+\frac{K_\delta(K_\delta-1)}2
=m+\frac{h_\delta(h_\delta+1)}2.
\end{equation}
Then $Q_{1+\delta,m}$ is well-defined through $n_\delta$ and
\begin{align}
Q_{1+\delta,m}(n)&=Q_{1,m}(n)
&& (1\le n<n_\delta),\label{eq:shadow-equality}\\
Q_{1+\delta,m}(n_\delta)&=K_\delta+1
&&\text{while }Q_{1,m}(n_\delta)=K_\delta.
\label{eq:first-departure}
\end{align}
Thus $n_\delta$ is the exact first departure time.  In particular,
\begin{equation}\label{eq:departure-scale}
n_\delta
=
\frac{1}{2\delta^2}+O_m(\delta^{-1})
\qquad (\delta\downarrow0).
\end{equation}
\end{theorem}

\begin{proof}
We verify the critical blocks in order.  Suppose that all preceding
supercritical values agree with the critical values.  Then the recursive
arguments and their unscaled arithmetic mean are exactly those in
Lemma~\ref{lem:branch-average}.

For $2\le k\le m$, write $B_m(s_k+a)=k-1+c/m$.  From
equation~\eqref{eq:c-count} and $a\le k-1$ we have
$0\le c\le m-k+1$.  Consequently,
\begin{equation}\label{eq:early-margin}
\frac{k-B_m(s_k+a)}{B_m(s_k+a)}
=
\frac{m-c}{m(k-1)+c}
\ge
\frac{k-1}{k(m-1)+1}
\ge
\frac1{2m-1}.
\end{equation}
Thus $(1+\delta)B_m(s_k+a)<k$.  Since
$B_m(s_k+a)\ge k-1$, its scaled floor is $k-1$, and the supercritical value
remains $k$.

For $k\ge m+1$, Lemma~\ref{lem:branch-average} gives
$B_m(s_k+a)=k-1$.  If $k<K_\delta$, then $k\le h_\delta$, so
\[
\delta(k-1)\le\delta(h_\delta-1)<1.
\]
It follows that $k-1\le(1+\delta)(k-1)<k$, and again the value remains $k$.
This proves agreement at every index preceding $s_{K_\delta}$ and also
proves well-definedness there.

At $n=s_{K_\delta}$, all earlier values still agree and the branch average
is $K_\delta-1=h_\delta$.  Hence
\[
Q_{1+\delta,m}(n)
=1+\floor{(1+\delta)h_\delta}
=K_\delta+\floor{\delta h_\delta}.
\]
By the definition of $h_\delta$,
$1\le\delta h_\delta<1+\delta<2$, so the last floor equals one.  Finally,
$h_\delta=\delta^{-1}+O(1)$ gives
equation~\eqref{eq:departure-scale}.
\end{proof}

\begin{remark}
The constant $(2m-1)^{-1}$ is sharp for complete shadowing.  For $m\ge2$,
at equality the value at $n=m+2$ is already $3$, whereas the critical value
is $2$.  For $m=1$, equality means $\delta=1$ and departure occurs at $n=2$.
\end{remark}

Theorem~\ref{thm:shadowing} shows why finite computations near $\alpha=1$
have exceptionally long critical transients.  It does not imply global
existence after the first departure.

\subsection{A floor-aware continuum heuristic}

Let
\[
A(n)=\frac1m\sum_{j=1}^{m}
Q_{\alpha,m}\bigl(n-Q_{\alpha,m}(n-j)\bigr).
\]
The recurrence can be written exactly as
\begin{equation}\label{eq:floor-defect}
Q_{\alpha,m}(n)=\alpha A(n)+\eta_n,
\qquad
\eta_n=1-\{\alpha A(n)\}\in(0,1],
\end{equation}
where $\{\cdot\}$ denotes fractional part.  This bounded term must be kept
when leading contributions cancel.

If $Q(n)\approx q(n)$ is slowly varying and the branches are asymptotically
equivalent, then
\[
A(n)\approx q(n-q(n))
\approx q(n)-q'(n)q(n).
\]
Writing $\alpha=1+\delta$ and replacing $\eta_n$ by an effective bounded
forcing term $\eta$ gives
\begin{equation}\label{eq:floor-aware-ode}
\alpha q(n)q'(n)\approx \delta q(n)+\eta.
\end{equation}
At criticality, Lemma~\ref{lem:branch-average} makes $A(n)$ integral for all
large $n$, hence $\eta_n=1$ exactly; equation~\eqref{eq:floor-aware-ode}
then recovers $q(n)\sim\sqrt{2n}$.  Above criticality the distribution of
$\eta_n$ is part of the dynamics, not a lower-order detail.  We therefore
use this calculation only as a guide, not as a proof of a scaling constant.

\subsection{The only possible limiting slope}

A linear ansatz $Q_{\alpha,m}(n)\sim cn$ with $0<c<1$ gives
\[
Q_{\alpha,m}\bigl(n-Q_{\alpha,m}(n-j)\bigr)
\sim c(1-c)n.
\]
Substitution into the recursion forces
\[
c=\alpha c(1-c),
\qquad
c=1-\frac1\alpha.
\]
The following proposition makes the necessity of this candidate rigorous
without assuming convergence in advance.

\begin{proposition}[Asymptotic slope constraint]\label{prop:slope}
Let $\alpha>1$, and suppose that equation~\eqref{eq:main-recursion} is
globally well-defined.  Put
\[
\ell=\liminf_{n\to\infty}\frac{Q_{\alpha,m}(n)}n,
\qquad
L=\limsup_{n\to\infty}\frac{Q_{\alpha,m}(n)}n.
\]
If $0<\ell\le L<1$, then
\begin{equation}\label{eq:slope-bracket}
\ell\le1-\frac1\alpha\le L.
\end{equation}
Consequently, if $Q_{\alpha,m}(n)/n$ converges to a limit in $(0,1)$, that
limit equals $1-\alpha^{-1}$.
\end{proposition}

\begin{proof}
Choose
\[
0<\varepsilon<\min\{\ell,1-L\}.
\]
For all sufficiently large $r$,
\begin{equation}\label{eq:lim-bounds}
(\ell-\varepsilon)r
\le Q_{\alpha,m}(r)
\le (L+\varepsilon)r.
\end{equation}
For $1\le j\le m$, set
\[
x_j(n)=n-Q_{\alpha,m}(n-j).
\]
Since $m$ is fixed, equation~\eqref{eq:lim-bounds} gives, uniformly in $j$,
\begin{equation}\label{eq:x-bounds}
(1-L-\varepsilon)n+O(1)
\le x_j(n)
\le(1-\ell+\varepsilon)n+O(1).
\end{equation}
In particular, $x_j(n)\to\infty$.  Applying
equation~\eqref{eq:lim-bounds} at $x_j(n)$ yields
\begin{align*}
Q_{\alpha,m}(x_j(n))
&\ge
(\ell-\varepsilon)(1-L-\varepsilon)n+O(1),\\
Q_{\alpha,m}(x_j(n))
&\le
(L+\varepsilon)(1-\ell+\varepsilon)n+O(1).
\end{align*}
Average over $j$, multiply by $\alpha$, and use the fact that both the
leading $1$ and the floor contribute only $O(1)$.  After division by $n$,
taking liminf and limsup and then letting $\varepsilon\downarrow0$, we get
\[
\ell\ge\alpha\ell(1-L),
\qquad
L\le\alpha L(1-\ell).
\]
Since $\ell,L>0$, division gives
\[
L\ge1-\frac1\alpha,
\qquad
\ell\le1-\frac1\alpha,
\]
which is equation~\eqref{eq:slope-bracket}.
\end{proof}

Proposition~\ref{prop:slope} identifies the only possible limiting density
in $(0,1)$.  It proves neither global existence nor convergence and does not
cover the boundary densities zero and one.

\begin{conjecture}\label{conj:supercritical}
For every fixed $m\ge2$, there exists $\varepsilon_m>0$ such that, for every
$1<\alpha<1+\varepsilon_m$, equation~\eqref{eq:main-recursion} is globally
well-defined and
\[
\frac{Q_{\alpha,m}(n)}n
\longrightarrow
1-\frac1\alpha.
\]
\end{conjecture}

The conjecture includes the assertion that the stated parameter set is an
interval; neither that interval structure nor its endpoint is presently
known.  We make no corresponding exclusion claim for $m=1$.

\section{Certified finite computations}\label{sec:numerics}

This section reports finite computations only.  It supplies evidence for
Conjecture~\ref{conj:supercritical}, not a proof.

For a rational parameter $\alpha=a/b$ in lowest terms, every recursive
update was evaluated using integer arithmetic:
\begin{equation}\label{eq:exact-update}
Q_{a/b,m}(n)
=
1+
\floor{
\frac{a}{bm}\sum_{j=1}^{m}
Q_{a/b,m}\bigl(n-Q_{a/b,m}(n-j)\bigr)
}.
\end{equation}
No floating-point number enters the recurrence or the breakdown test.  A
run stops if any requested index lies outside $\{1,\ldots,n-1\}$.

Table~\ref{tab:slopes} gives results for $m=3$ at $N=200{,}000$.  All four
runs remained well-defined through $N$.

\begin{table}[htbp]
\centering
\caption{Exact-arithmetic finite computations for $m=3$ at
$N=200{,}000$.  The final column is a finite-run status, not a claim of
global existence.}
\label{tab:slopes}
\begin{tabular}{@{}ccccc@{}}
\toprule
$\alpha$ & exact value & $1-\alpha^{-1}$ & $Q(N)/N$ & status \\
\midrule
$1.05$ & $21/20$ & $0.047619$ & $0.048020$ & no failure to $N$ \\
$1.10$ & $11/10$ & $0.090909$ & $0.091145$ & no failure to $N$ \\
$1.25$ & $5/4$   & $0.200000$ & $0.200130$ & no failure to $N$ \\
$1.50$ & $3/2$   & $0.333333$ & $0.333420$ & no failure to $N$ \\
\bottomrule
\end{tabular}
\end{table}

To probe the dependence on the averaging length, Table~\ref{tab:cross-m}
repeats three parameters for $m=2,3,4,5$.  Every one of these twelve runs
also remained well-defined through $N=200{,}000$.

\begin{table}[htbp]
\centering
\small
\caption{Cross-$m$ exact-arithmetic check at $N=200{,}000$.  Entries are
$Q_{\alpha,m}(N)/N$; no run failed before the cutoff.}
\label{tab:cross-m}
\begin{tabular}{@{}cccc@{}}
\toprule
$m$ & $\alpha=11/10$ & $\alpha=5/4$ & $\alpha=3/2$ \\
\midrule
$2$ & $0.091140$ & $0.200110$ & $0.333525$ \\
$3$ & $0.091145$ & $0.200130$ & $0.333420$ \\
$4$ & $0.091145$ & $0.200115$ & $0.333400$ \\
$5$ & $0.091145$ & $0.200110$ & $0.333420$ \\
\midrule
$1-\alpha^{-1}$ & $0.090909$ & $0.200000$ & $0.333333$ \\
\bottomrule
\end{tabular}
\end{table}

Figure~\ref{fig:certified} displays the same runs on a logarithmic horizontal
axis.  Dashed horizontal lines mark the forced candidate slopes from
Proposition~\ref{prop:slope}.  The slow finite-size drift is most visible at
$\alpha=21/20$, in agreement with the exact shadowing scale from
Theorem~\ref{thm:shadowing}.

\begin{figure}[htbp]
\centering
\includegraphics[width=0.94\textwidth]{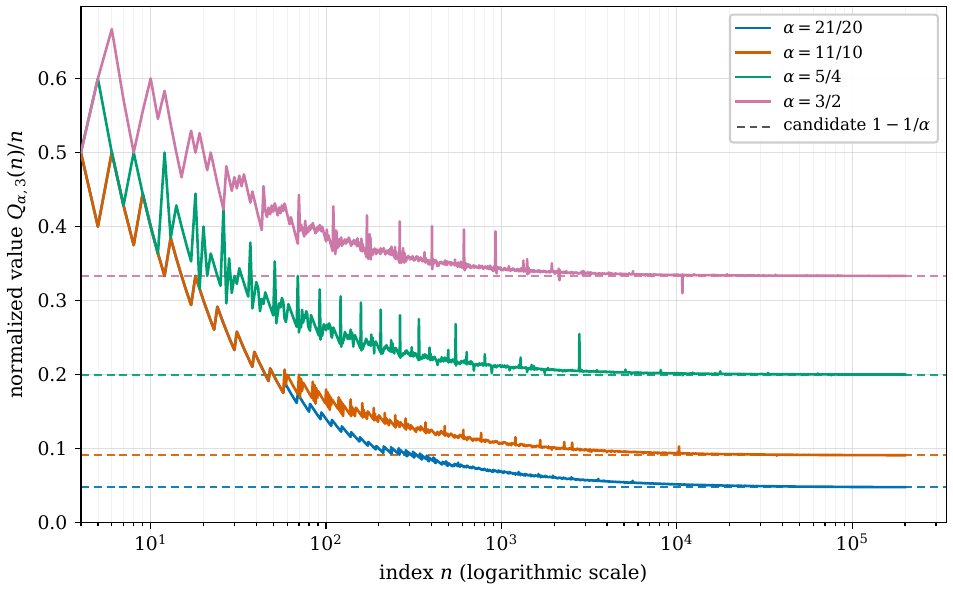}
\caption{Exact-rational computations of $Q_{\alpha,3}(n)/n$ through
$N=200{,}000$.  Solid curves are the computed ratios; dashed lines are the
candidate limits $1-\alpha^{-1}$.  Only a logarithmically thinned set of
computed points is drawn.}
\label{fig:certified}
\end{figure}

The script \texttt{compute\_experiments.py}, distributed with the source,
reproduces Tables~\ref{tab:slopes} and~\ref{tab:cross-m}, the plotted data,
the failure checks, and Figure~\ref{fig:certified}.  The reported files were
generated with Python~3.14.4 and Matplotlib~3.10.9 by running
\[
\texttt{python compute\_experiments.py}
\]
in the source directory.  Before the main experiment, the script performs
exact regression tests of Theorem~\ref{thm:aggregator} for the arithmetic,
quadratic, harmonic, minimum, and a positively weighted quasi-arithmetic
mean.  It also checks the critical closed form and a deliberately triggered
large-parameter breakdown.  All recurrence, floor, aggregator-test, and
breakdown decisions in these checks are integer or rational operations.

The apparent distinction between positive and nonpositive power means
reported in v1 was a floating-point artifact.
Theorem~\ref{thm:aggregator} corrects and supersedes that numerical claim.
Floating evaluation immediately before a floor is unsafe at integer
boundaries.

\section{Conclusions and open problems}\label{sec:conclusion}

Finite arithmetic averaging preserves the canonical solution of Golomb's
triangular recursion exactly: after the initial block, changing $m$ only
translates the sequence.  More generally, this invariance is a consequence
of the local floor-lock condition~\eqref{eq:floor-lock}, not of the arithmetic
mean specifically.  Thus harmonic, geometric, negative-power, and other
strictly internal means belong to the same exact critical class.  For
$m\ge2$, the maximum lies outside that class but has the explicit shifted
block law of Theorem~\ref{thm:maximum}, with the same leading square-root
growth.  For $m=1$, it coincides with the identity aggregator and is already
covered by Theorem~\ref{thm:aggregator}.

The parameter $\alpha=1$ is rigorously the boundary of the frozen solution.
Theorem~\ref{thm:shadowing} further shows that a slightly supercritical orbit
shadows the critical sequence for a time of order $(\alpha-1)^{-2}$.  Beyond
that time, global existence and convergence remain open.  The numerical
evidence is compatible with a linear regime, while
Proposition~\ref{prop:slope} determines its only possible limiting slope.

We close with five concrete problems.
\begin{enumerate}[leftmargin=2.1em]
\item \textbf{Supercritical existence and convergence.}
Prove or disprove Conjecture~\ref{conj:supercritical}; ideally obtain a
quantitative error term around $(1-\alpha^{-1})n$.

\item \textbf{Parameter survival geometry.}
For fixed $m$, classify the set of $\alpha$ for which the recursion is
global.  In particular, determine whether the component immediately to the
right of $1$ is nontrivial and whether the full survival set is an interval.

\item \textbf{The floor crossover.}
Study the double scaling $\alpha=1+\delta$ and
$n\asymp\delta^{-2}$.  The exact critical shadow ends on this scale, and the
subsequent fractional parts in equation~\eqref{eq:floor-defect} may govern a
nontrivial transition and logarithmic corrections.

\item \textbf{Boundary aggregators.}
Classify non-strict aggregators not covered by
Definition~\ref{def:floor-admissible}, including medians, quantiles, and means
with zero weights.  Determine which preserve the standard shifted blocks,
which admit alternative exact laws as the maximum does for $m\ge2$, and which
preserve only square-root growth.

\item \textbf{Initial data and perturbations.}
Determine whether the critical frequency law is stable under changes of the
initial block or bounded perturbations, even when monotonicity and contiguous
blocks are lost.
\end{enumerate}

\section*{Data and code availability}

The exact-arithmetic reproducibility archive for
Section~\ref{sec:numerics} is published as version 1.0.0 on Zenodo
\cite{mantovanelli-code}.  It contains the script
\nolinkurl{compute_experiments.py}, the CSV summary, the vector figure,
machine-readable citation metadata, and an open-source license.  The script
regenerates the reported outputs without stochastic input.

\section*{Acknowledgments}

OpenAI language-model tools were used for assistance with LaTeX formatting,
language editing, literature discovery, and code generation.  The author is
responsible for verifying the mathematical arguments, computations, and
final text.

\begingroup
\footnotesize
\setlength{\bibsep}{1pt plus 0.2pt}
\bibliographystyle{plainnat}
\bibliography{references}
\endgroup

\end{document}